\documentclass[10pt,twoside]{article}
\usepackage{lgrind}
\usepackage{fullpage, amsfonts, longtable}
\usepackage[pdftex]{graphicx,color}
\usepackage{amsmath}
\allowdisplaybreaks[1]
\usepackage{amsthm}
\usepackage{epsfig}
\usepackage{setspace}
\usepackage{float}
\usepackage{amsthm}
\usepackage{amssymb}
\usepackage{multirow}
\usepackage{array}
\usepackage{mdwlist}
\usepackage{natbib}
\usepackage{subfigure}
\renewcommand{\baselinestretch}{1}
\def\blfootnote{\xdef\@thefnmark{}\@footnotetext}
\long\def\symbolfootnote[#1]#2{\begingroup%
\def\thefootnote{\fnsymbol{footnote}}\footnote[#1]{#2}\endgroup}
\newcommand{\Caption}[2][Figure]{{\caption[{#1}]{\footnotesize {#2}}}}
\newcommand{\defn}[1]{{\textup{\textbf{#1}}}}
\newcommand{\rank}[1]{{\textup{rank}(#1)}}

\newcommand{\R}{\mathbb{R}}

\begin{document}

\vfill

\begin{center}
{\Large Characteristics of Optimal Solutions to the Sensor Location Problem}

\singlespacing \vspace{1in} David. R. Morrison

University of Illinois at Urbana-Champaign

Champaign, IL

\vspace{0.5in}

Susan. E. Martonosi

Harvey Mudd College

301 Platt Blvd.

Claremont, CA 91711

(909) 607-0481

martonosi@hmc.edu

\vspace{1 in}

Submitted for publication on October 3, 2010

\end{center}

\vfill

\noindent \textbf{Abstract:} In~\cite{odmatrixest}, the authors present the Sensor Location Problem: that of locating the minimum number of traffic sensors at intersections of a road network such that the traffic flow on the entire network can be determined.  They offer a necessary and sufficient condition on the set of monitored nodes in order for the flow everywhere to be determined.  In this paper, we present a counterexample that demonstrates that the condition is not actually sufficient (though it is still necessary).  We present a stronger necessary condition for flow calculability, and show that it is a sufficient condition in a large class of graphs in which a particular subgraph is a tree.  Many typical road networks are included in this category, and we show how our condition can be used to inform traffic sensor placement.

\newtheorem{combcutset}{Definition}[section]
\newtheorem{slpdef}[combcutset]{Definition}
\newtheorem{badslpthm}{Theorem}[section]
\newtheorem{boundvertex}{Definition}[section]
\newtheorem{adjvert}[boundvertex]{Definition}
\newtheorem{turnfactor}[boundvertex]{Definition}
\newtheorem{adjmat}{Definition}[section]
\newtheorem{incidmat}[adjmat]{Definition}
\newtheorem{linind}[adjmat]{Theorem}
\newtheorem{reducedlinind}[adjmat]{Theorem}
\newtheorem{canmakesquare}[adjmat]{Theorem}
\newtheorem{canmakesquarecor}[adjmat]{Corollary}
\newtheorem{bpath}{Definition}[section]
\newtheorem{slpconj}[bpath]{Conjecture}
\newtheorem{revslpthm}[bpath]{Theorem}
\newtheorem{slpdecdef}{Definition}[section]
\newtheorem{domset}[slpdecdef]{Definition}
\newtheorem{npcomplete}[slpdecdef]{Theorem}
\newtheorem{treethmcount}{Definition}[section]
\newtheorem{treethm}[treethmcount]{Theorem}
\newtheorem{treecor}[treethmcount]{Corollary}
\newtheorem{treelemma}[treethmcount]{Lemma}
\newtheorem{treelemma2}[treethmcount]{Lemma}
\doublespacing

\section{Introduction}

Traffic congestion is a significant problem in most major cities in the world.  An important first step in mitigating road congestion is to know the distribution of cars on each road of the network.  This can be achieved by using traffic sensors to count cars traveling into and out of an intersection. However, placing sensors on every intersection is not only prohibitively expensive, it is also inefficient:  if some sensors were removed, traffic flow through those intersections might still be calculated by applying flow conservation laws and knowledge of the fraction of cars turning in each direction at each intersection.  In fact, even a vertex cover is inefficient for these reasons.  Thus, we want to locate the minimum number of sensors such that we can still determine the distribution of cars in the entire
network.  This problem was introduced in \cite{odmatrixest} and named the Sensor Location Problem, (SLP).

\cite{odmatrixest} present a necessary and sufficient condition on the set $M$ of monitored intersections such that the traffic flow on the entire network is calculable.  We present a counterexample demonstrating that the condition, while necessary, is not sufficient.  Using the insights provided by the counterexample, we develop a stronger necessary condition that, while not sufficient in general, is sufficient in a large class of networks in which a particular \textit{unmonitored subgraph}, to be defined in this paper, is a tree.  We present several examples of road networks, including the standard grid network, to which this sufficient condition can be used to confirm that the flow can be completely specified.  Moreover we present examples for which the condition is \textit{not} sufficient, but where the failure of the necessary condition also provides useful information about the network.

First, we review the terminology and notation used in \cite{odmatrixest}.  Then, we present our counterexample in Section \ref{sec:slpexs}, and develop a matrix representation for the problem in Section \ref{sec:slpmatrix}.  Section \ref{sec:neccond} derives a graph-theoretic necessary condition for flow calculability, and Section \ref{sec:tree} demonstrates that this condition is sufficient in the case when each unmonitored subgraph is a tree.  In Section \ref{sec:examples} we provide examples of how this new condition could be used for decision support by traffic engineers.  We offer concluding remarks in Section \ref{sec:conc}.

\section{Definitions}\label{sec:defs}

  Let the road network be represented by a directed graph $G = (V,A)$, where $V$ is a set of intersections and $A$ is a set of ``two-way'' directed arcs (roads).  That is, if $u,v \in V$ and  $uv \in A$, then $vu  \in A$, but the traffic flow on arc $uv$ need not equal that on arc $vu$.  We represent the traffic flowing over the roads by a network flow
function $f:A \to \R$ that satisfies the flow conservation law at each vertex
$v \in V$:

\begin{equation}\label{eqn:fcl}
\sum_{e \in v^-} f_e - \sum_{e \in v^+} f_e + S_v = 0,
\end{equation}
\noindent
where $v^-$ is the set of arcs with head at $v$,  $v^+$ is
the set of arcs with tail at $v$, and $S_v$ is the \defn{balancing flow} at vertex
$v$.  The sources and sinks of traffic, called \defn{centroids}, are the vertices with non-zero balancing flows; the set of all such centroids we denote $B$.  Because flow is conserved at each vertex, we have $\sum_{\substack{v \in V}}  S_v = 0$.  We assume that while the set $B$ is known, the values of the balancing flows for vertices in $B$ are unknown.

To determine the network flow function $f$, sensors are
placed at various intersections in the road network.  We denote the set of
\defn{monitored vertices} by $M$.  If an intersection is monitored, then the number of cars entering and leaving the intersection along each
road connected to the intersection is revealed.  We denote the set of vertices directly adjacent to vertices in $M$ via an arc in $A$ as $A(M)$.

We finally assume knowledge of the \defn{turning ratios} at every intersection in the
network.   The turning ratio $c_{vu}$ for arc $vu$ at vertex $v$ is simply the percent of incoming traffic to $v$ that leaves along arc $vu$.  That is,
\begin{equation}\label{eqn:turn1}
f_{vu} = c_{vu} \sum_{e \in v^-} f_e.
\end{equation}
Define the \defn{turning factor} of arc $vu$ with respect to given reference arc $vw$ to be the ratio of their turning ratios:
\begin{equation}\label{eqn:turnfact}
\alpha_{vu} = \frac{c_{vu}}{c_{vw}}.
\end{equation}  Then we can write the flow $f_{vu}$ of any outgoing arc $vu$ from $v$ in terms of $f_{vw}$ as
\begin{equation}\label{eqn:turn3}
f_{vu} = \alpha_{vu} f_{vw}.
\end{equation}
The values for the turning ratios can be obtained from historical data about traffic patterns if
available, or can be determined easily by monitoring existing traffic patterns for a short time.

When a set $M$ of vertices is monitored, the flow on all arcs between vertices in $M$ and between $M$ and $A(M)$ are known, as well as the balancing flows at each centroid in $M$.  Applying the turning ratios, we also know the flow on all arcs between vertices in $A(M)$.  We call the set of arcs connecting vertices in $M$ and $A(M)$, on which the flow can be computed directly from monitoring $M$ and applying turning ratios, the \defn{combined cutset} of $M$:
\begin{combcutset}[\cite{odmatrixest}]\label{def:combcutset}
The \defn{combined cutset of $M$}, $C_M$, is the set of arcs in the subgraph
of $G$ induced by $M \cup A(M)$.
\end{combcutset}
\noindent As an aside, we can also use the turning ratios to determine outgoing flow from vertices in $A(M)$ to vertices neither in $A(M)$ nor $M$; these arcs are not part of the combined cutset, but will be used later.

We are now ready to define the Sensor Location Problem (SLP):

\begin{slpdef}[{Sensor Location Problem, \cite{odmatrixest}}]
\label{def:slp}
Given a two-way directed graph $G = (V,A)$, a network flow
function $f$ and a set of centroids $B$, what is the smallest set $M$ of
monitored vertices such that knowledge of all turning ratios, the values of
$f$ on incoming and outgoing arcs of $M$ and balancing flows $S_v$ on $M$ uniquely determines $f$ and the balancing flows $S_v$ everywhere on $G$?
\end{slpdef}

We focus on the verification version of SLP and seek a condition to verify that a proposed set $M$ uniquely determines $f$ and the balancing flows.

\section{A proposed condition and counterexample}\label{sec:slpexs}

We see from the definition of the combined cutset of $M$ that these are arcs over which the problem of determining the flow has already been solved directly from monitoring.  Thus, we can remove $C_M$ from the graph, and try to use the
remaining flow coming out of $A(M)$, turning ratios and flow balance equations to determine the flow everywhere else in the
graph.  We therefore define the \defn{unmonitored subgraph} of $G$ to be the subgraph $G^{'}$ that remains when $C_M$ has been removed from the graph: $G^{'}=(V-M, A-C_M)$.  This subgraph contains all arcs over which the flow is not completely determined by monitoring.

The unmonitored subgraph is often, but not always, disconnected.  We call the $i^{th}$ connected component of the unmonitored subgraph the \defn{$i^{th}$ unmonitored component} and label it $G^{'}_i$.  We label the set of centroids in that
component $B_i$, and the set of (originally) adjacent vertices in
that component $A_i(M)$.  \cite{odmatrixest} present a proof of the
following condition on the set $M$ in order for the flow function $f$ to be uniquely determined.
While this is a necessary condition, we present an example that demonstrates it is not actually sufficient in general.

\begin{badslpthm}[\cite{odmatrixest}]\label{thm:badslp}
Given a set of monitored vertices $M$, the flow on a digraph $G$ can be
uniquely determined everywhere if and only if for every unmonitored component $G_i$ of $G$,
\[|B_i| \leq |A_i(M)|.\]
\end{badslpthm}

In their proof of this theorem, the authors compare the number of unknown arc and balancing flow variables to the number of flow balance and turning ratio equations when this condition holds.  They argue (correctly) that the number of equations must be at least the number of unknowns, which happens only if $|B_i| \leq |A_i(M)|$.  However, in their argument that the condition is sufficient, they neglect the possibility that some of the resulting equations might be linearly dependent, and thus the solution will not be unique.

To see this, consider the following example (shown in Figure~\ref{fig:slpcex1}).  Let $\delta^+(u)$ be the outgoing degree of vertex $u$, and suppose that the turning ratios $c_{uv} = 1/\delta^+(u)$ for all arcs $uv$ (the flows on all outgoing arcs from vertex $u$ are equal).
By monitoring vertex
$a$, the unmonitored subgraph $G'$ induced by removing the combined cutset has only a single connected component, consisting of the vertices $b, c, d, e, f$ and the arcs between them.  $A(M)$ in this component is $\{b,d\}$, and $B-M = \{e,f\}$.  Thus $|A(M)| =
|B - M| = 2$, and by Theorem~\ref{thm:badslp}, we should be able to
determine $f$ and the vector $S$ of balancing flows uniquely.

\begin{figure}[h!]
\begin{center}
\includegraphics[scale=0.7]{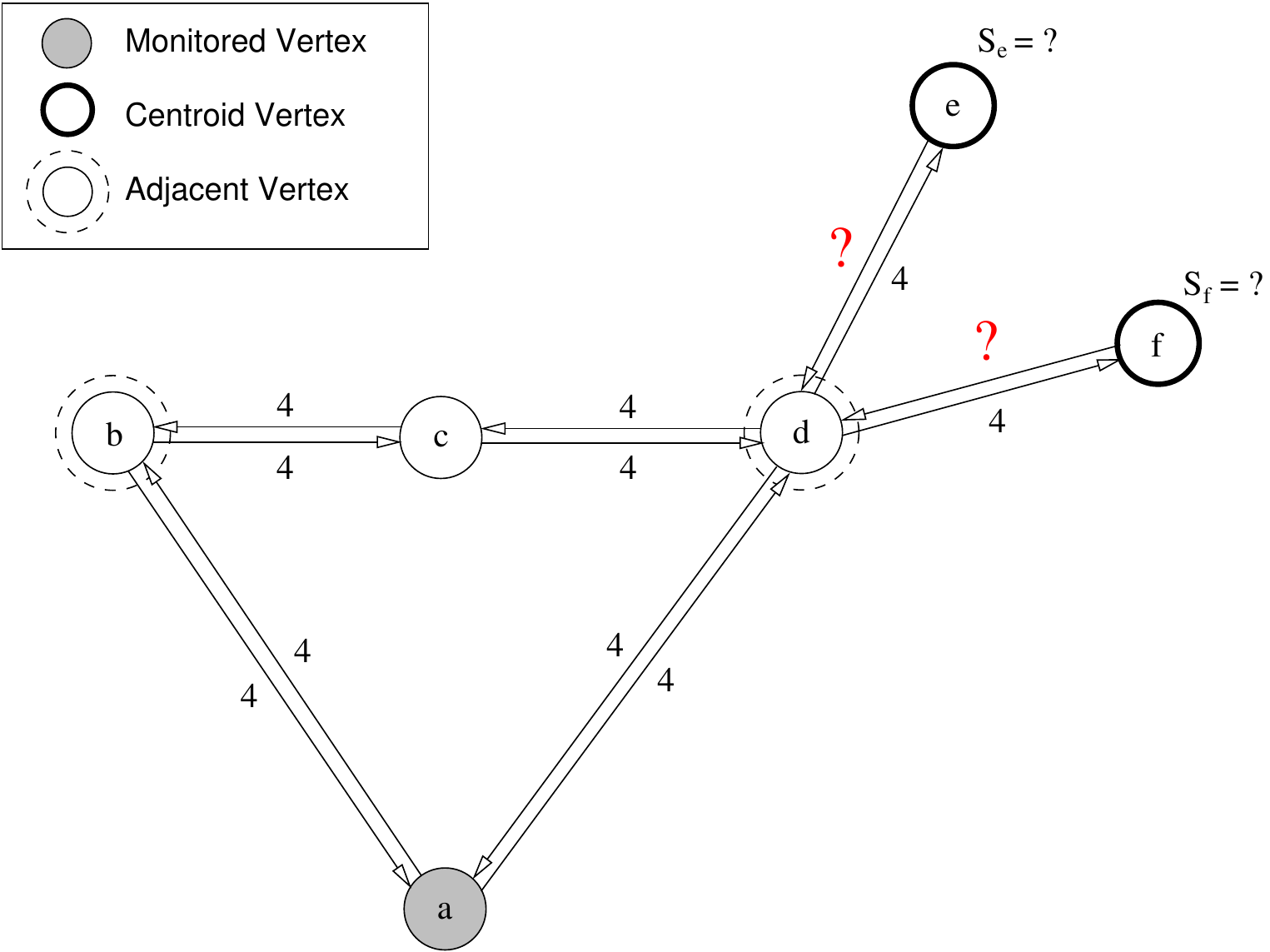}
\Caption{\label{fig:slpcex1}
A counterexample to the flow calculation theorem
(Theorem~\ref{thm:badslp}).  If we monitor vertex $a$ in the above graph, the graph with cutset $C_M = \{ab, ba, ad, da\}$ removed satisfies the conditions in Theorem~\ref{thm:badslp}.
However, we cannot calculate $f_{ed}$ or $f_{fd}$ from the known information.}
\end{center}
\end{figure}

However, suppose we observe $4$ units of flow along arcs $ab,ba,ad,$ and $da$.
 We apply the flow balance equation and knowledge of the turning ratios sequentially at each vertex until we get stuck.  Consider vertex $b$.  It is not a centroid, so $S_b=0$, and since flows on all outgoing arcs are equal, $f_{bc}=f_{ba}=4$.  To preserve balance of flow, $f_{cb}=4$ as well.  By a similar logic, we obtain $f_{cd}=f_{dc}=4$ at vertex $c$ and $f_{de}=f_{df}=4$ at vertex $d$.  We cannot determine $f_{ed}$ and $f_{fd}$ because both $e$ and $f$ are centroids, and their balancing flows are unknown.  Balancing flows in the network must sum to zero, so $S_f = -S_e$, leaving us with the following system of equations having three unknowns and three equations, as predicted by Theorem~\ref{thm:badslp}.  \begin{eqnarray}
\begin{array}{ccccccc}
f_{ed}&+&f_{fd}& & &=& 8\\
f_{ed}& &      &-&S_e&=&4\\
      & &f_{fd}&+&S_e&=&4
\end{array}\end{eqnarray}  Notice, however, that these equations are linearly \textit{dependent} and thus fail to admit a unique solution.  Therefore, the condition provided in Theorem~\ref{thm:badslp} is not sufficient.

Unfortunately, there are many such counterexamples, including cases when the graph is a tree or when the inequality in the theorem is strict.  Fortunately, the subsequent work of \cite{odmatrixest} and \cite{combinatorialSLP} is correct despite the erroneous Theorem~\ref{thm:badslp}.  Nonetheless, it is still valuable to understand why the theorem is incorrect and to formulate a new theorem that guarantees the calculability of traffic flows on a monitored graph.  To better understand the circumstances under which Theorem~\ref{thm:badslp} fails, we next examine the problem via the graph's incidence matrix.

\section{SLP and Invertible Matrices}\label{sec:slpmatrix}
Let $\mathbf{E}$ be the $|V| \times |A|$ incidence matrix where the $(u,e)^{th}$ entry is $-1$ if vertex $u$ is the tail of arc $e$, $1$ if it is $e$'s head, and $0$ if $e$ is not incident
to $u$.  Let $\mathbf{f}$ be the $|A|$-length vector of unknown arc flows and $\mathbf{S}$
the $|V|$-length vector of balancing flows.  The system of linear flow conservation constraints at each vertex then takes the form \begin{equation}\label{eqn:matrix1}
\mathbf{E}\ \mathbf{f} + \mathbf{S} = \mathbf{x},
\end{equation} where $\mathbf{x}=\mathbf{0}$.  Notice that the sum of these equations yields the balancing flow constraint $\sum_{\substack{u \in V}}\mathbf{S}_u = 0$, so we do not need to add this constraint to system (\ref{eqn:matrix1}).

This system does not include the known turning ratios or the observed flow along monitored arcs, and thus
contains more unknown variables than are necessary.  We can therefore reduce this system of equations to a more compact representation, as follows:

\begin{enumerate}
\item For each vertex $u \in V$, we designate an arbitrary outgoing arc $e_u$ to be the \defn{canonical arc} for vertex $u$.  Since we know the turning ratios of the graph, the flow over any arc $uv$ is $f_{uv} = \alpha_{uv}f_{e_u}$.  This reduces the number of flow variables from $|A|$ to $|V|$, and we can modify the unknown flow vector $\mathbf{f}$ to include only the $|V|$ canonical arcs.

\item Having expressed the flow on any arc $uv$ in terms of the flow on $e_u$, the flow balance matrix $\mathbf{E}$ collapses into a square matrix $\mathbf{\hat{E}}$, where row $u$ still corresponds to the balance equation at vertex $u$, and column $v$ corresponds to the canonical arc for vertex $v$, $e_v$.  The $(u,v)^{th}$ entry of $\mathbf{\hat{E}}$ is given by
\[\mathbf{\hat{E}}_{u,v} = \left\{\begin{array}{cl}
\alpha_{vu} & \textrm{ if $u$ and $v$ are connected}\\
-\sum_{\substack{w \mbox{ adjacent to }u}} \alpha_{uw} & \textrm{ if $u=v$}\\
0 & \textrm{ if $u$ and $v$ are not connected}
\end{array}\right.\]

\item\label{list:balflowcol} We also augment $\mathbf{\hat{E}}$ with $|B|$ columns for the unknown balancing flows at the centroids.  The column corresponding to the centroid at vertex $u$ has a $1$ in the $u^{th}$ row and $0$'s everywhere else.  Likewise, we create a single $(|V|+|B|)$-length vector
    $\mathbf{g} = \left[\begin{array}{c} \mathbf{f} \\ \mathbf{S} \end{array} \right]$
    of unknown canonical arc and balancing flows. Equation (\ref{eqn:matrix1}) then becomes
\begin{equation}\label{eqn:matrix2}
\mathbf{\hat{E}}\mathbf{g} = \mathbf{x},
\end{equation} where $\mathbf{x}$ is still the zero vector.

\suspend{enumerate}
We next incorporate the known flow values obtained by monitoring vertices in $M$.

\resume{enumerate}
\item\label{list:removemonrows} For each vertex $m \in M$, the flow along $m$'s canonical arc and the balancing flow (if $m$ is a centroid) are known.  We can remove row $m$ from the matrix $\mathbf{\hat{E}}$.  We also remove column $m$, corresponding to vertex $m$'s canonical arc.  Next, we update the right-hand side vector $\mathbf{x}$ with the known flow values by subtracting $f_{e_m}$ times the removed $m^{th}$ column from $\mathbf{x}$.  This is equivalent to subracting $\alpha_{mu} f_{e_m}$ from the $u^{th}$ entry of $\mathbf{x}$ for each vertex $u$ adjacent to $m$.  If $m$ is a centroid, we also remove the column of $\mathbf{\hat{E}}$ corresponding to its balancing flow.  We likewise remove the entry from $\mathbf{g}$ corresponding to $f_{e_m}$ (and $S_m$ if $m$ is a centroid), and remove the $m^{th}$ entry from $\mathbf{x}$.

 \item\label{list:adjvertrows} For each vertex $a \in A(M)$, the outgoing flow from $a$ to any vertex $m \in M$ is
 monitored, so by turning ratios, we can deduce the flow over $a$'s canonical arc.  We therefore remove column
 $a$ from $\mathbf{\hat{E}}$ and subtract $f_{e_a} = \frac{1}{\alpha_{am}}f_{am}$ times column $a$ from the right-hand
 side vector $\mathbf{x}$.  This is equivalent to subtracting $\alpha_{au}f_{e_a}$ from the $u^{th}$ entry of $\mathbf{x}$ for each vertex $u$ adjacent to $a$ and adding $\sum_{\substack{w \mbox{ adjacent to } a}}\alpha_{aw}f_{e_a}$ to the $a^{th}$ entry of $\mathbf{x}$.  We remove the entry from $\mathbf{g}$ corresponding to $f_{e_a}$.

We name the resulting coefficient matrix for the system of equations the \defn{flow calculation matrix} $\mathbf{F}$ and rewrite for the last time our original system of equations
\begin{equation}\label{eqn:matrix3}
\mathbf{F} \mathbf{g} = \mathbf{x}.
\end{equation}
If equation (\ref{eqn:matrix3}) has a unique solution (which occurs when the columns of $\mathbf{F}$ are linearly independent), then we can uniquely
determine the flow everywhere on the graph.

\end{enumerate}

\begin{figure}[t!]
\begin{center}
\includegraphics[scale=0.7]{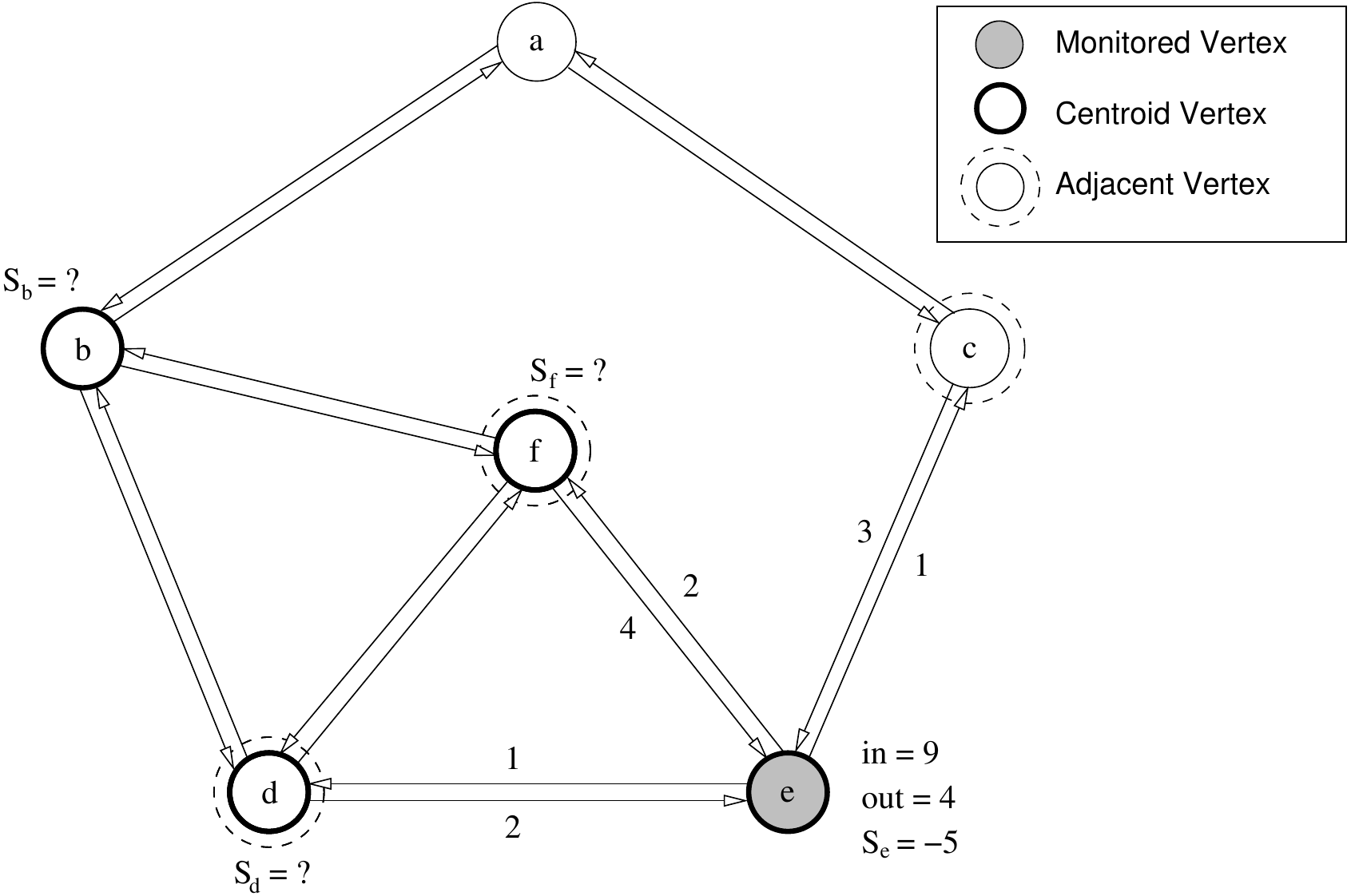}
\Caption{\label{fig:slppenta1} A network in which the set of centroids is $B = \{b,d,e,f\}$, and vertex $e$ is monitored, revealing the flows indicated on the arcs into and out of $e$.  In this case, we can calculate the flow everywhere on the graph, as demonstrated by equation (\ref{eqn:matex}) having a unique solution.}
\end{center}
\end{figure}

For example, consider the graph in Figure~\ref{fig:slppenta1}, with $M = \{e\}$ and flows on monitored arcs as indicated in the figure.  We choose arcs $ab, ba, ca, db, ed$ and $fb$ to be our canonical representatives for each vertex. We also assume that all turning ratios are equal except at vertex
$e$ (so $\alpha_{uv}=1$ for all $uv \neq e$), where monitoring has revealed the turning factors to be $\alpha_{ef} = 2$ and $\alpha_{ec} = 1$.  The corresponding reduced system of equations is:

\begin{equation} \label{eqn:matex}
{\small
\left(\begin{array}{rrrrrrrr}
-2 & 1 &  0 & 0 & 0 \\
1 & -3 &  1 & 0 & 0 \\
1 & 0 &  0 & 0 & 0 \\
0 & 1 &  0 & 1 & 0 \\
0 & 1 &  0 & 0 & 1 \\
\end{array}\right)
\left(\begin{array}{r}
f_{ab} \\ f_{ba} \\ S_b \\ S_d \\ S_f
\end{array}\right) =
\left(\begin{array}{r}
-3 \\ -6 \\ 5 \\ 1 \\ 8
\end{array}\right)}
\end{equation}

It is easy to check that $\rank{\mathbf{F}} = 5$, and
thus the columns are linearly independent; this implies that
equation (\ref{eqn:matrix3}) is solvable for the graph in
Figure~\ref{fig:slppenta1}.

\section{A new necessary condition}\label{sec:neccond}
The flow is uniquely calculable if and only if the matrix $\mathbf{F}$ has full column rank.  An obvious necessary (but insufficient) condition is for $\mathbf{F}$ to have at least as many rows as columns.  $\mathbf{F}$ has $|V|-|M|-|A(M)| + |B - M|$ columns and $|V|-|M|$ rows.  Therefore, we require $|B-M| \leq |A(M)|$.  The necessary condition proved in \cite{odmatrixest} is stronger: $|(B-M)_i| \leq |A(M)_i|$ for all connected components $i$ in the unmonitored subgraph induced by removing the arcs in the combined cutset.  In fact, we can prove an even stronger necessary condition that relies solely on the topology of the graph.  This condition correctly identifies that the example of Figure~\ref{fig:slpcex1} will not yield a unique solution, whereas the original condition $|(B-M)_i| \leq |A(M)_i|$ could not predict this.

Our difficulty in calculating the flow arose when we reached vertex $d$.  Although we knew the flow exiting vertex $d$ along the arcs toward $e$ and $f$, we were unable to determine the flow entering $d$ because both $e$ and $f$ were centroids, contributing unknown balancing flows.  Traffic originating or terminating at vertices $e$ and $f$ got ``mixed up'' at vertex $d$ and could not be uniquely differentiated.

This observation leads us to define a $\mathbf{B}$-path, which we use to correct Theorem~\ref{thm:badslp}.
\begin{bpath}
A \defn{$\mathbf{B}$-path} is a path starting at a centroid and ending
at a vertex in $A(M)$.
\end{bpath}

\noindent This is a similar, but less restrictive, definition than that given for \defn{MB-feasible} paths in \cite{combinatorialSLP}.  Using this definition, we present the following theorem, which provides a stronger necessary condition for flow calculability.

\begin{revslpthm}[Statement A]\label{thm:revslp}
Let $G = (V,A)$ be a two-way directed graph with centroid set $B$, and let
$M$ be a set of monitored vertices.  The flow on arcs in $G$ and the balancing flow at the vertices in $B$ can be
uniquely determined everywhere only
if there exists a set $\mathcal{P}$ of $|B - M|$ vertex disjoint $B$-paths.
\end{revslpthm}
\addtocounter{bpath}{-1}

This is a stronger necessary condition than that given in Theorem~\ref{thm:badslp} because it is not satisfied by our counterexample in Figure~\ref{fig:slpcex1}.  We see in Figure~\ref{fig:ex2paths} that any set of two $B$-paths will be
forced to intersect at vertex $d$.  Thus, the
number of disjoint $B$-paths is smaller than $|B - M|$ and we are unable to
calculate the flow.

\begin{figure}[t!]
\begin{center}
\includegraphics[scale=0.7]{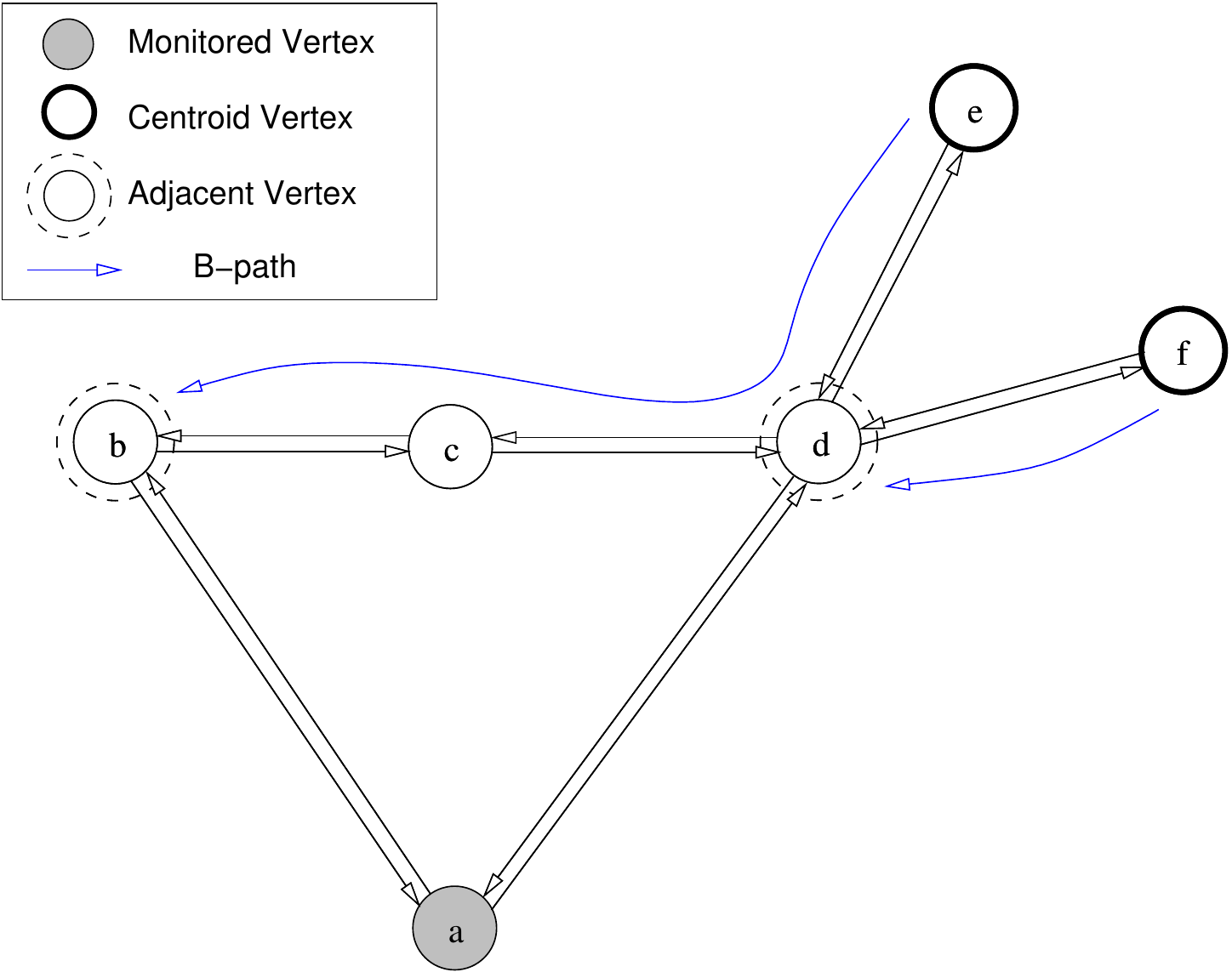}
\Caption[A case in which there does not exist a set of disjoint $B$-paths]
{\label{fig:ex2paths}
The graph from Figure~\ref{fig:slpcex1}), together with a set of
$B$-paths.  However, any two $B$-paths must pass through vertex $d$, so there
is no set of $|B - M|$ disjoint $B$-paths associated with $M$.}
\end{center}
\end{figure}

To prove Theorem~\ref{thm:revslp}, we must translate its statement related to the topological structure of the network into our algebraic framework described earlier.  We note first that the number of vertex-disjoint
$B$-paths cannot be larger than the size of a minimum disconnecting set $C$
between $B-M$ and $A(M)$, by Menger's theorem.  Thus, we require $|C| \geq |B-M|$.  (In fact, the size of the minimum disconnecting set will never strictly exceed $|B-M|$).

Next, we partition the graph $G$ into its unmonitored components by removing the combined cutset.  If $u$ and $v$ are in different partitions in the graph, then there was no path from $u$ to $v$ in $G$ except through $M$ or along an $a_1a_2$ edge for some $a_1$ and $a_2 \in A(M)$.  Because all rows and columns corresponding to $M$ and all columns corresponding to $A(M)$ have been removed from the
matrix, vertex $u$'s flow balance equation will not include any $e_v$ or $S_v$ terms, and $e_u$ and $S_u$ will not appear in vertex $v$'s flow balance equation.  Thus, we can rearrange the flow calculation matrix $\mathbf{F}$ into block form by collecting rows and columns corresponding to vertices in each unmonitored component, and prove the theorem for each component independently.  We rephrase our original theorem accordingly:

\begin{revslpthm}[Statement B]
Let $G, B,$ and $M$ be as in Theorem~\ref{thm:revslp} (Statement A), with the graph partitioned into unmonitored components and the flow
calculation matrix partitioned into blocks as described.  For each unmonitored component $i$, let $C_i$ be the minimum vertex cut between $(B - M)_i$ and $A(M)_i$.  (If $(B-M)_i$ is empty, then let $C = \varnothing$).
$\rank{\mathbf{F}^i} = \#\{\textrm{columns of } \mathbf{F}^i\}$ (and hence the flow on $G^{'}_i$ is calculable) only if $|C_i| = |(B - M)_i|$.
\end{revslpthm}
\begin{proof}

For ease of notation, we drop the subscript $i$ and henceforth refer
to all sets in the context of a given unmonitored component $i$.  We assume the
component contains at least one centroid, otherwise the theorem is
true trivially because both $B-M$ and $C$ are empty.  Within
component $i$, we call $V_M$ the set of vertices that are not in $M$
or $A(M)$ and are connected to $M$ by some path that does not pass
through $C$ (i.e. they are on the $M$ side of the cut $C$).
Similarly, we call $V_B$ the set of vertices not in $B - M$ that are
on the $B - M$ side of the cut. Note that $C, A(M),$ and $B - M$
could all overlap, as shown in Figure~\ref{fig:sets1}; we label
these intersections as shown, where $X_{A(M),C} = (A(M) \cap C)
\backslash (B-M)$, $X_{A(M)} = A(M) \backslash (C \cup (B-M))$, etc.
Note that since $C$ is by definition a vertex cut between $B - M$
and $A(M)$, the set $X_{A(M),B-M} = (A(M) \cap (B-M)) \backslash C$
is empty.

\begin{figure}[h!]
\begin{center}
\includegraphics[scale=0.5]{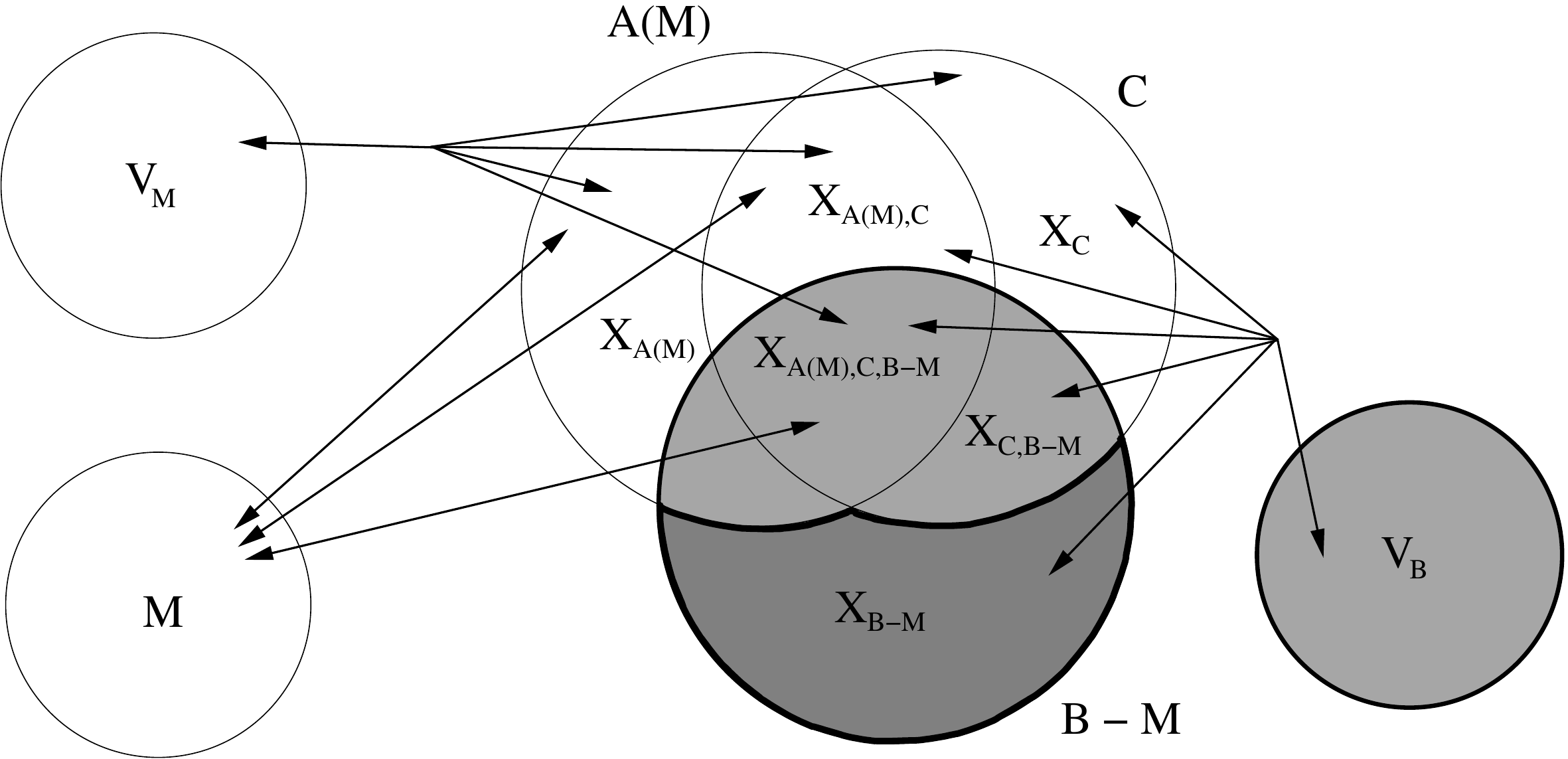}
\Caption[The partition of the vertex set of the
graph]{\label{fig:sets1} The partition of the vertex set for Theorem
\ref{thm:revslp}.  $V_M$ is the set of unaccounted-for vertices on
the $M$ side of the cut, and $V_B$ is the set of unaccounted-for
vertices on the $B - M$ side of the cut.  Bold arrows indicate
possible connections between sets.  Some of these sets may be
empty---in particular, note that by definition, there can be no
vertices in $((B - M) \cap A(M)) \setminus C$, since $C$ must
separate $B - M$ and $A(M)$.  The shaded-in regions correspond to
the columns included in the submatrix $\mathbf{F}^{*}$.}
\end{center}
\end{figure}

Let us consider a submatrix $\mathbf{F}^{*}$ of $\mathbf{F}$ that
contains only the columns corresponding to canonical arcs for
vertices in $X_{B-M}$ and $V_B$ and to balancing flows at vertices
in $X_{B-M}$, $X_{C, B-M}$ and $X_{A(M), C, B-M}$.  These are the
shaded regions of Figure~\ref{fig:sets1}. Since $\mathbf{F}$ has
linearly independent columns, $\mathbf{F}^{*}$ has full column rank, and
\begin{equation}\label{eqn:rowzerocol}
K \leq R - Z,
\end{equation}
where $K$ is the number of columns of $\mathbf{F}^{*}$, $R$ the number of rows and $Z$ the number of zero rows.  By construction,
\[K = |X_{B-M}|+|V_B|+|B-M|\]
and
\[R = |X_{A(M)}|+|X_{A(M),C}|+|X_{A(M),C,B-M}|+|X_{B-M}| + \]
\[|X_C|+|X_{C,B-M}|+|V_M|+|V_B|.\]

Next we determine $Z$.  As we see in Figure~\ref{fig:sets1}, there
are no arcs from vertices in $V_M$ or $X_{A(M)}$ to vertices in
$X_{B-M}$ or $V_B$ by definition of the cut $C$.  Moreover, vertices
in $V_M$ and $X_{A(M)}$ are not centroids.  Therefore, the rows in
$\mathbf{F}^{*}$ corresponding to vertices in $V_M$ or $X_{A(M)}$ are
all zero, and $Z \geq |V_M|+|X_{A(M)}|$.  Applying inequality
(\ref{eqn:rowzerocol}) and canceling common terms, we see that
$|B-M| \leq |X_C| + |X_{A(M), C}| + |X_{C, B-M}| + |X_{A(M), C,
B-M}| = |C|$.  Because $|C|$ can never exceed $|B-M|$, we have $|C|
= |B-M|$.
\end{proof}

As an example, we walk through the construction of the $\mathbf{F}^{*}$-matrix for
our original counterexample of Figure~\ref{fig:slpcex1}.  We start with the flow calculation matrix $\mathbf{F}$ in the system of equations $\mathbf{F} \mathbf{g} = \mathbf{x}$:
\begin{equation}
{\small\begin{array}{rr@{}rrrrr@{}l}
b: & \multirow{5}{*}{$\left(\begin{array}{@{}r}\\\\\\\\\\\end{array}\right.$}
& 1 & 0 & 0 & 0 & 0 &
\multirow{5}{*}{$\left.\begin{array}{@{}l}\\\\\\\\\\\end{array}\right)$}\\
c: & & -2 & 0 & 0 & 0 & 0 &\\
d: & & 1 & 1& 1 & 0 & 0 & \\
e: & & 0 & -1 & 0 & 1 & 0 &\\
f: & & 0 & 0 & -1 & 0 & 1&
\end{array}
\left(\begin{array}{r}
f_{cb} \\ f_{ed} \\ f_{fd} \\ S_e \\ S_f
\end{array}\right) =
\left(\begin{array}{r}
4 \\ -8 \\ 12 \\ -4 \\ -4
\end{array}\right)}\end{equation}

When we remove the cutset associated with monitored vertex $a$, the minimum vertex cut between $A(M) = \{b,d\}$ and $B-M = \{e,f\}$ is $C = \{d\}$.  $|C| \neq |B-M|$, so we should not be able to calculate the flow.  We generate the $\mathbf{F}^{*}$ submatrix using the sets $X_{A(M)}=\{b\}$, $V_M = \{c\}$, $X_{A(M),C} = \{d\}$, and $X_{B-M} = \{e, f\}$:

\begin{equation}
\mathbf{F}^{*} =
{\small\begin{array}{rr@{}rrrr@{}l}
& & ed & fd & S_e & S_f &\\
b: & \multirow{5}{*}{$\left(\begin{array}{@{}r}\\\\\\\\\\\end{array}\right.$}
& 0 & 0 & 0 & 0 &
\multirow{5}{*}{$\left.\begin{array}{@{}l}\\\\\\\\\\\end{array}\right)$}\\
c: & &  0 & 0 & 0 & 0 \\
d: & &  1 & 1 & 0 & 0 \\
e: & & -1 & 0 & 1 & 0\\
f: & & 0 & -1 & 0 & 1
\end{array}}\end{equation}
Notice that the first two rows of the matrix are $0$, which means that the
rank of this submatrix can't be any higher than $3$.  This implies that the
rank of $\mathbf{F}$ cannot equal the number of columns, and thus the flow
on the graph cannot be calculated.

\section{A sufficient condition for trees}
\label{sec:tree}
%%%%%%%%%%%%%%%% DISCUSSION OF SUFFICIENCY %%%%%%%%%%%%%%%%%%%%%%%
Next we turn to the question of sufficiency. Unfortunately, the condition is not sufficient for graphs in general, but is sufficient in the case of networks whose unmonitored components are all trees.  Figure~\ref{fig:bpathcex1} provides an example of a general graph in which there are $|B - M|$ vertex-disjoint $B$-paths, but the matrix $E_M$ still does not have linearly independent columns and the flow on the graph cannot be calculated.\begin{figure}
\begin{center}
\includegraphics[scale=0.7]{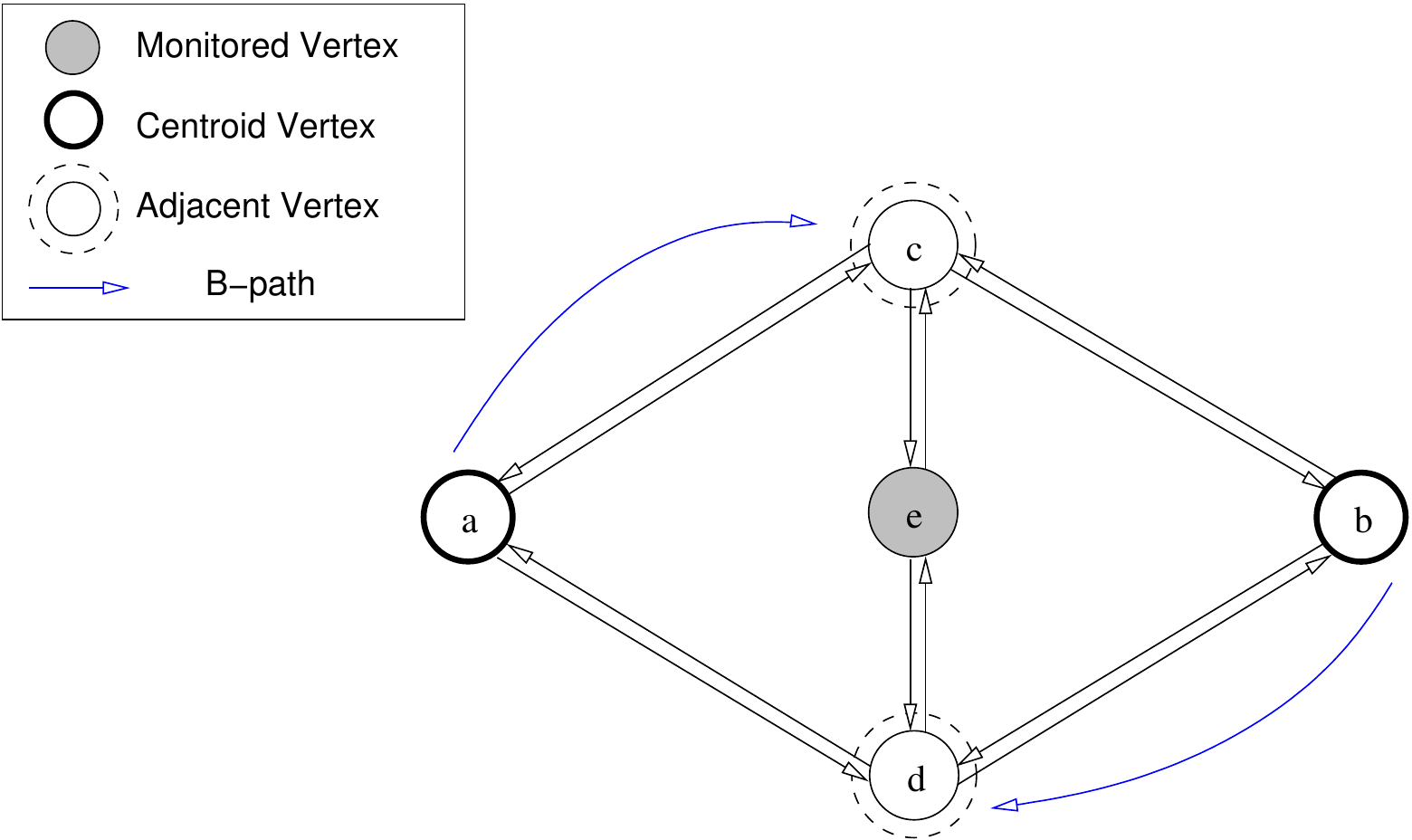}
\Caption[A graph in which we cannot calculate the
flow]{\label{fig:bpathcex1} In this graph, monitoring vertex $e$
creates $2$ disjoint $B$-paths, satisfying the conditions of Theorem~\ref{thm:revslp}.  However, the $\mathbf{F}$ matrix does not
have linearly independent columns, and thus we cannot calculate the flow on the
graph.}
\end{center}
\end{figure}  We see that the unmonitored subgraph of this example, which we obtain by removing $M$'s combined cutset (in this case, the arcs $ce, ec, de,$ and $ed$), is not a tree.  However, the following theorem states that as long as the unmonitored components of a graph are all trees, our condition is sufficient to guarantee the calculability of traffic flow.

\begin{treethm}\label{thm:treethm}
Let $G, B,$ and $M$ be as in Theorem~\ref{thm:revslp}, with the flow
calculation matrix partitioned into blocks as described.  For each unmonitored component $i$,
let $C_i$ be the minimum vertex cut between $(B - M)_i$ and $A(M)_i$.  If the $i^{th}$ component is a tree, then $\rank{\mathbf{F}^i} = \#\{\textrm{columns of } \mathbf{F}^i\}$ (that is, the flow on block $i$ is calculable) if and only if $|C_i| = |(B - M)_i|$.
\end{treethm}

Prior to proving this theorem, we first prove the following lemma:
\begin{treelemma} \label{thm:treelemma}
Let $G$ be a two-way directed tree with known turning ratios, containing no centroids and having root vertex $r$.  Suppose that $G$ is attached at $r$ to a graph $\hat{G}$ at vertex $v$ with known flow value $f_{vr}$.  Then $f_{rv}=f_{vr}$ and the flow on $G$ can be determined.
\end{treelemma}
\begin{proof}
We prove this by induction on the number $n$ of vertices in $G$.  As a base case, suppose $n=1$, then $G$ contains only the leaf node $r$.  Because $r$ is not a centroid, its balancing flow is zero, so $f_{rv}=f_{vr}$, and the flow on $G$ has been determined.

Suppose the statement is true for any tree of size strictly less than $n$ and let $G$ have size $n$. Consider the vertex $r \in G$ which is attached to graph $\hat{G}$ at vertex $v$.  Because $G$ is a tree with no centroids, flow is conserved, so $f_{rv}=f_{vr}$, and all other outgoing flows of $r$ can be determined using turning ratios.  Moreover, $r$ is connected to $\mbox{deg}(r)$ subtrees of $G$ each of size strictly less than $n$ and having root vertices $v_i, i=1...\mbox{deg}(r)$ with known incoming flow values $f_{rv_i}$.  So the flow on each subtree can also be determined, and $f_{v_ir}=f_{rv_i}$.  We have therefore found the flow on the entire graph.
\end{proof}

Thus, on a subtree with no centroids, the flow can be computed knowing only a single incoming arc to the tree.
We now prove Theorem~\ref{thm:treethm}.

\begin{proof}
Theorem~\ref{thm:revslp} handles the necessary condition.  Let $n = |(B - M)_i|$.  Because $|C_i| = |(B-M)_i|$, $|A(M)_i|$ must also be at least $n$, and there exists a pairing between a subset of $A(M)_i$ and the vertices in $(B-M)_i$ such that the set of paths between all pairs $a_i \in A(M)_i$ and $b_i \in (B-M)_i$ are vertex disjoint.  (If $|A(M)_i| > n$, then the ``extra'' vertices in $A(M)_i$ will act like centroids with \textit{known} balancing flow equal to the incoming flow from $M$ and can be treated like any other non-centroid.)  We will induct on $n$ to show how to propagate the flow calculation through the graph.

\begin{figure}[h!]
\begin{center}
\includegraphics[scale=0.75]{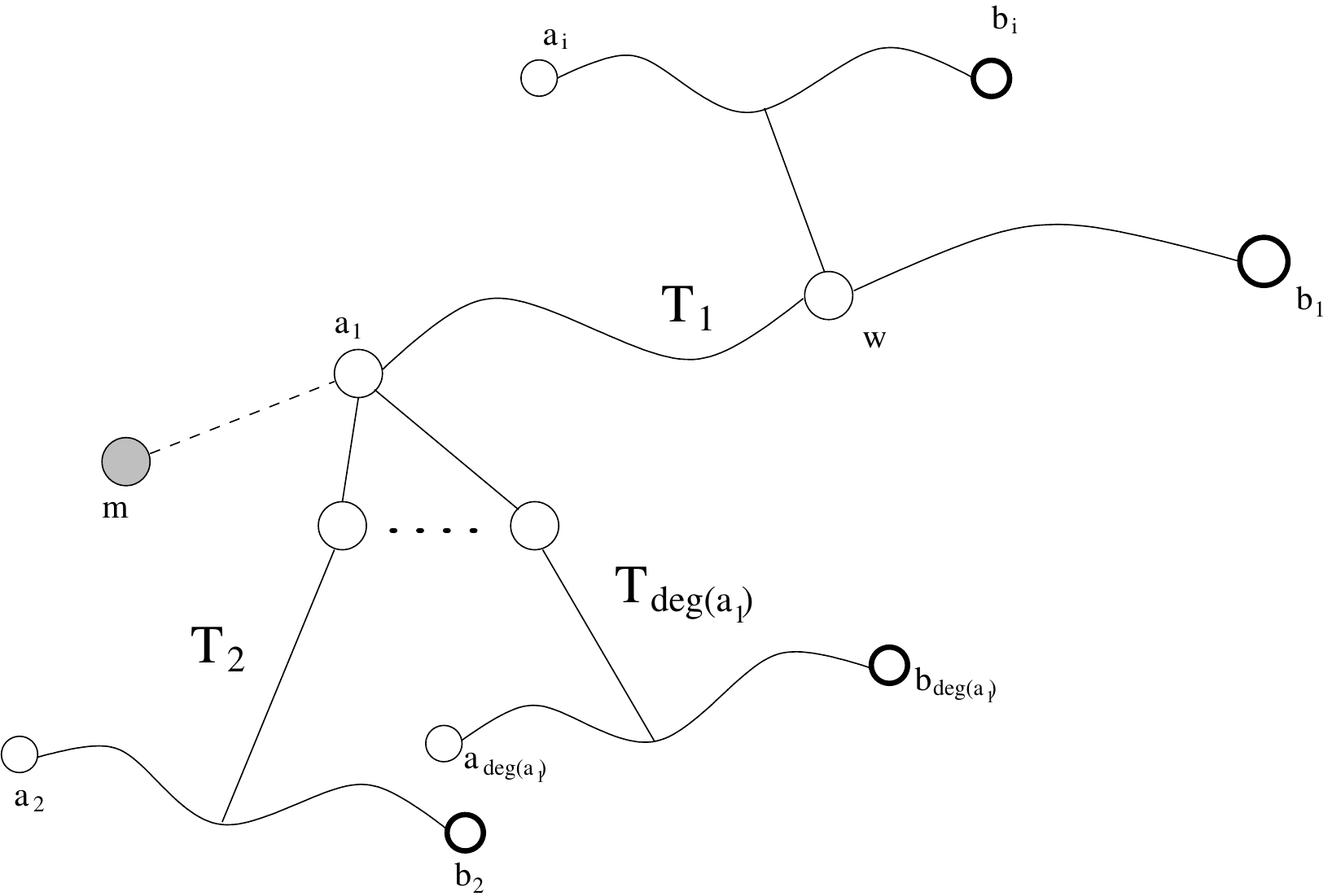}
\Caption{\label{fig:treefig} The tree structure induced by the pairing of centroids and adjacent
vertices in Theorem~\ref{thm:treethm}.  Note that no path from a vertex $a_j \in A(M)$ to its pair
$b_j \in B - M$ can cross any other vertex in $A(M)$.  This allows us to treat each subtree
separately when calculating the flow along it.}
\end{center}
\end{figure}

If $n=0$, then our partition consists of a tree having no centroids that, in the original graph, is connected to a vertex in $M$.  The flow along an incoming arc to $G^{'}_i$ from $M$ is known due to monitoring, so our Lemma~\ref{thm:treelemma} tells us that the flow along every arc in $G^{'}_i$ can be determined.

Suppose the theorem is true for any partition having $|(B-M)_i|<n$ unmonitored centroids, and consider a partition having $|(B-M)_i|=n$.  Without loss of generality, consider centroid $b_1$ and its matching adjacent vertex $a_1$.  Let $T_1, ..., T_{deg(a_1)}$ be the subtrees of $T$ rooted at $a_1$'s neighbors $t_1, ..., t_{deg(a_1)}$, and assume that $T_1 \cup \{a_1\}$ is the tree containing the $(a_1, b_1)$ pair (See Figure~\ref{fig:treefig}).  Every subtree maintains the original pairing of vertices in $A(M)_i$ and $(B-M)_i$ because these pairings corresponded to vertex-disjoint paths which could not pass through $a_1$.  Moreover, each subtree $T_j, j \neq 1$ has strictly fewer than $n$ such pairings and satisfies the induction hypothesis.  The flow on these subtrees can be calculated.  It remains to determine the flow on the edges between $a_1$ and $t_j$: $f_{a_1t_j}$ is an outgoing flow from $a_1$ which is known by monitoring and application of turning ratios.  $f_{t_ja_1}$ can be expressed in terms of the flow on $t_j$'s canonical edge.  Thus, there is still a unique solution for the flow on $T_j \cup \{a_1\}$, and all incoming flow to $a_1$ from trees $T_2, ..., T_{deg(a_1)}$ has been determined.

Next we consider the tree $T_1 \cup \{a_1\}$, by propagating flow calculations along the path from $a_1$ to $b_1$.  The outgoing flow from $a_1$ along the path is known by applying turning ratios.  If $a_1 = b_1$, then in fact $T_1$ is empty.  $a_1$'s incoming and outgoing flows can be used to find its balancing flow, and the flow on $G^{'}_i$ has been completely determined.  If $a_1 \neq b_1$, the incoming edge to $a_1$ from the next vertex on the path to $b_1$ is the only edge incident to $a_1$ whose flow has not yet been determined; it can be determined by flow conservation. By a similar logic, we can propagate these flow calculations along the path from $a_1$ to $b_1$ until we reach the first vertex $w$ having degree greater than 2.  All outgoing flows of $w$ are known by applying turning ratios to the known outgoing flow from $w$ to the vertex preceding $w$ in the path.  Each branch of $w$ that does not contain $b_1$ is once again a subtree that maintains the (strictly fewer than $n$) original $(B-M)_i$ and $A(M)_i$ pairings.  By our induction hypothesis, the flow on this branch can be determined and by the same reasoning as above, the flow on the edges between $w$ and the branch can also be determined.  Flow conservation determines the incoming flow to $w$ from the next vertex on the path to $b_1$.  We continue this way until vertex $b_1$ is reached.  All outgoing flows from $b_1$ are determined by applying turning ratios to the known outgoing flow from $b_1$ into the vertex preceding $b_1$ in the path.  The flow on branches stemming from $b_1$ can be determined by our induction hypothesis, and the balancing flow at $b_1$ is simply the difference between all outgoing and incoming flows at $b_1$.  The flow on the tree has been calculated. \end{proof}

An obvious corollary is the following:

\begin{treecor}
If $G$ is a two-way directed tree with centroid set $B$ and monitored vertex set $M$, the flow on $G$ can be calculated if and only if there exist at least $|B-M|$ vertex-disjoint paths between $A(M)$ and $B-M$.
\end{treecor}

This suggests that it is the presence of cycles in unmonitored subgraphs that can occasionally lead to difficulties in uniquely determining the flow.

\section{Applications to traffic sensor placement} \label{sec:examples}
While it might at first seem restrictive to require the unmonitored subgraph to be a tree in order to guarantee flow calculability, we show in this section that this is not the case.  In fact, a broad collection of road networks can have trees as unmonitored subgraphs.  Moreover, even for those networks whose unmonitored subgraphs contain cycles, our condition can still provide useful information about the placement of traffic sensors.

Consider the traffic network shown in Figure~\ref{fig:grid2}.
\begin{figure}[ht]
\centering
\subfigure[Original road network]{
\includegraphics[height=2.4in]{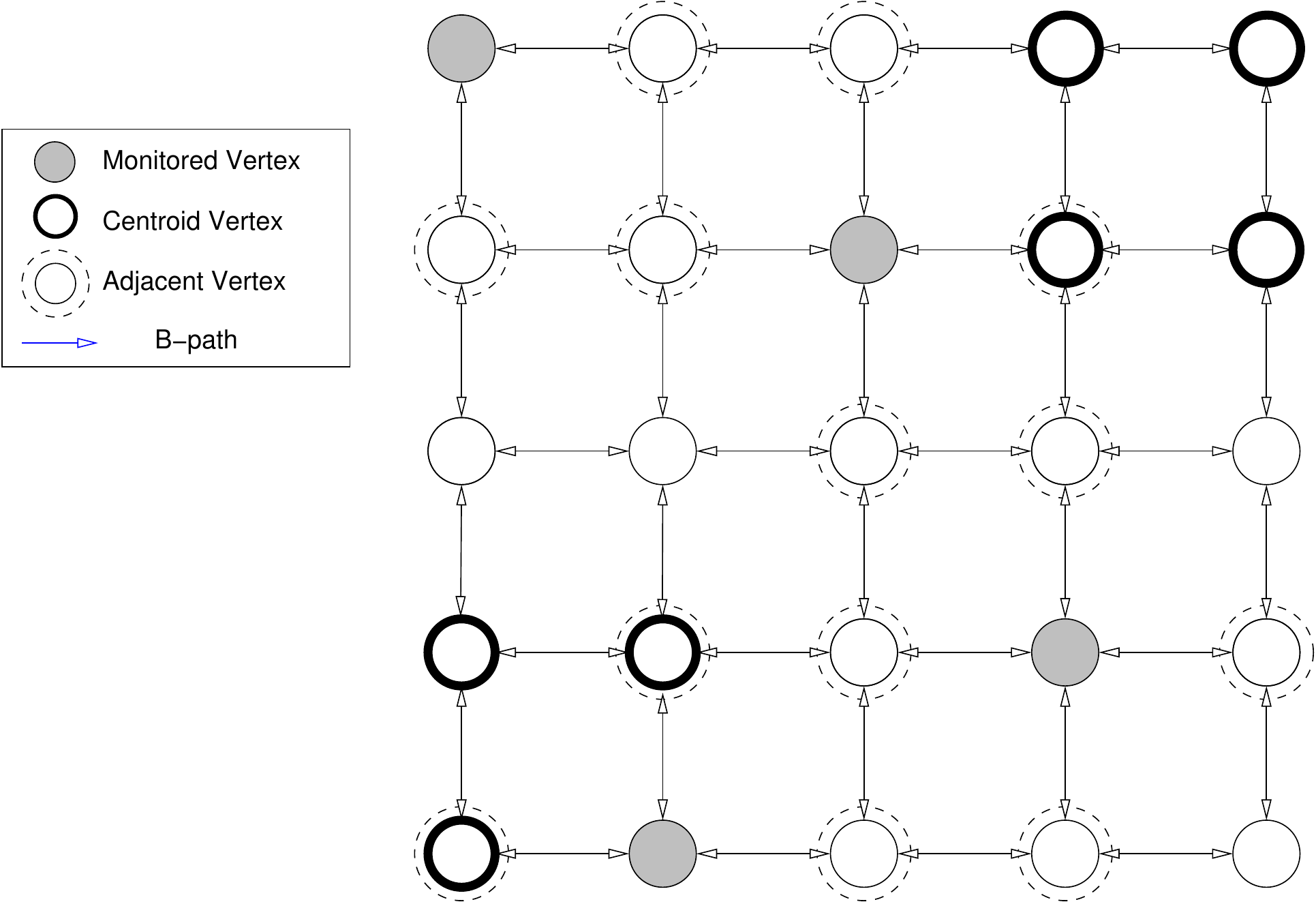}
\label{fig:grid2}
}
\subfigure[Unmonitored subgraph obtained by removing the combined cutset $C_M$]{
\includegraphics[height=2.4in]{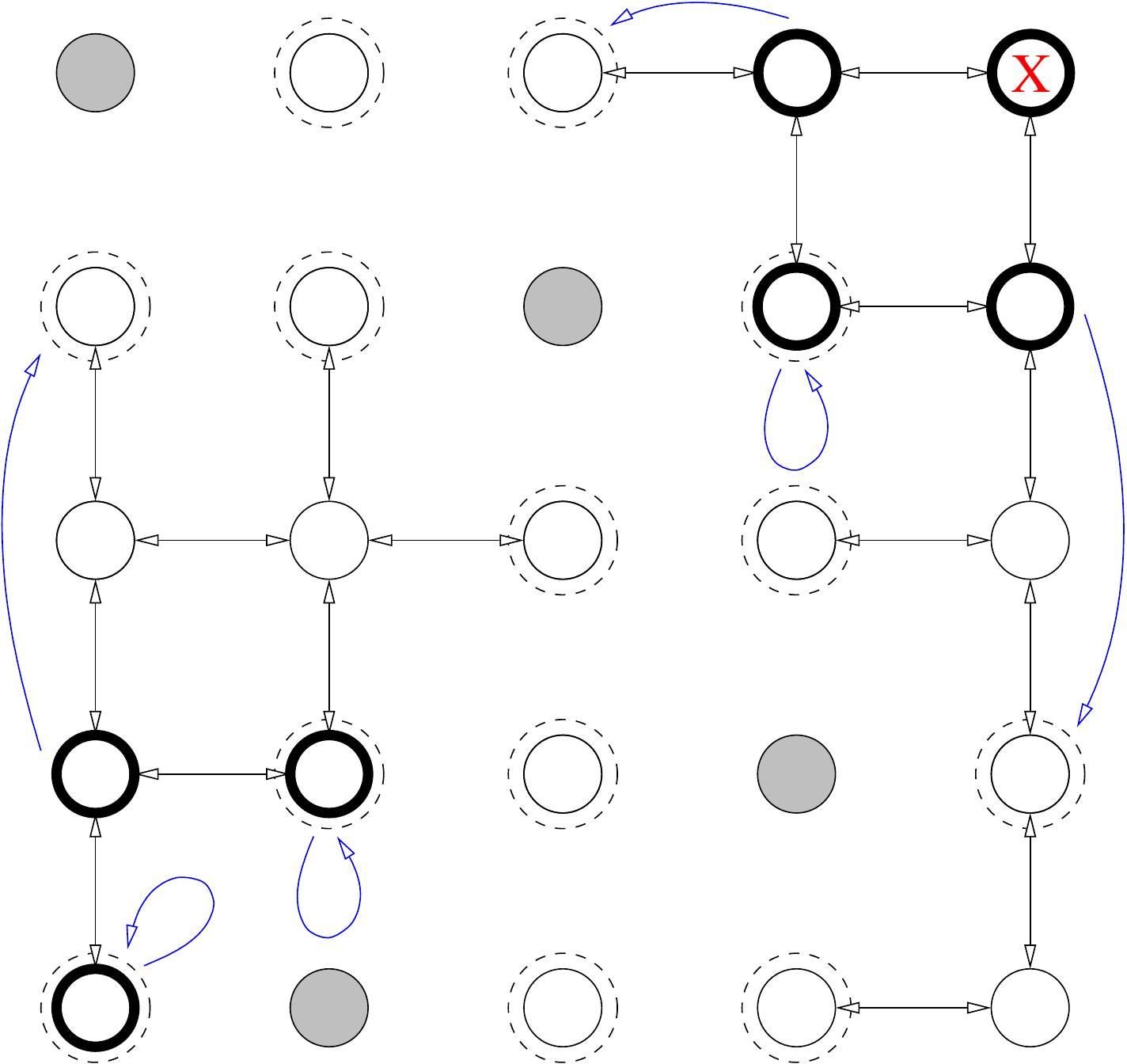}
\label{fig:grid2rem}
}
%\label{fig:grid2Example}
\caption{\label{fig:grid2Example} A road network with four monitored vertices (shaded) and seven centroid vertices (bold).  The unmonitored subgraph has two components.  The leftmost component has three centroids and five vertices in $A(M)$, satisfying the necessary condition of Theorem~\ref{thm:badslp}.  All centroids have a corresponding B-path, but because this unmonitored component is not a tree, we cannot be certain that the flow is calculable on this region of the network.  The rightmost component has four centroids and five vertices in $A(M)$, satisfying the necessary condition of Theorem~\ref{thm:badslp}.   However, the centroid labeled $X$ does not have its own B-path in this unmonitored component, and hence the traffic flow is not calculable in this region of the network, by Theorem~\ref{thm:revslp} (Statement A).}
\end{figure}
This is a traditional grid network found in many cities, and it has twenty-five intersections, of which seven are considered to be centroids.  A traffic planner might be interested in monitoring four intersections on the network to calculate the traffic flow throughout it.  If she places the monitors on the four vertices shaded in Figure~\ref{fig:grid2}, then the unmonitored subgraph is as shown in Figure~\ref{fig:grid2rem}.  The condition of Theorem~\ref{thm:badslp} that $|B-M|_i \leq |A(M)_i|$ is satisfied on both unmonitored components, but the rightmost component violates our necessary condition of Theorem~\ref{thm:revslp} (Statement A) because the centroid marked by $X$ does not have its own B-path.  Therefore, we are able to conclude that flow is not calculable in this region of the network.  The leftmost unmonitored component satisfies our necessary condition, but because this component is not a tree, we are unable to conclude that the flow is calculable in this region.

If the traffic planner simply rearranges two of the monitored vertices, as shown in Figure~\ref{fig:grid3Example}, she is easily able to construct a network whose unmonitored subgraph is a tree.  The sufficient condition of Theorem~\ref{thm:treethm} is satisfied, and she can conclude that by placing the monitors in this orientation, the traffic flow on the network will be completely calculable.
\begin{figure}[ht]
\centering
\subfigure[Original road network]{
\includegraphics[height=2.5in]{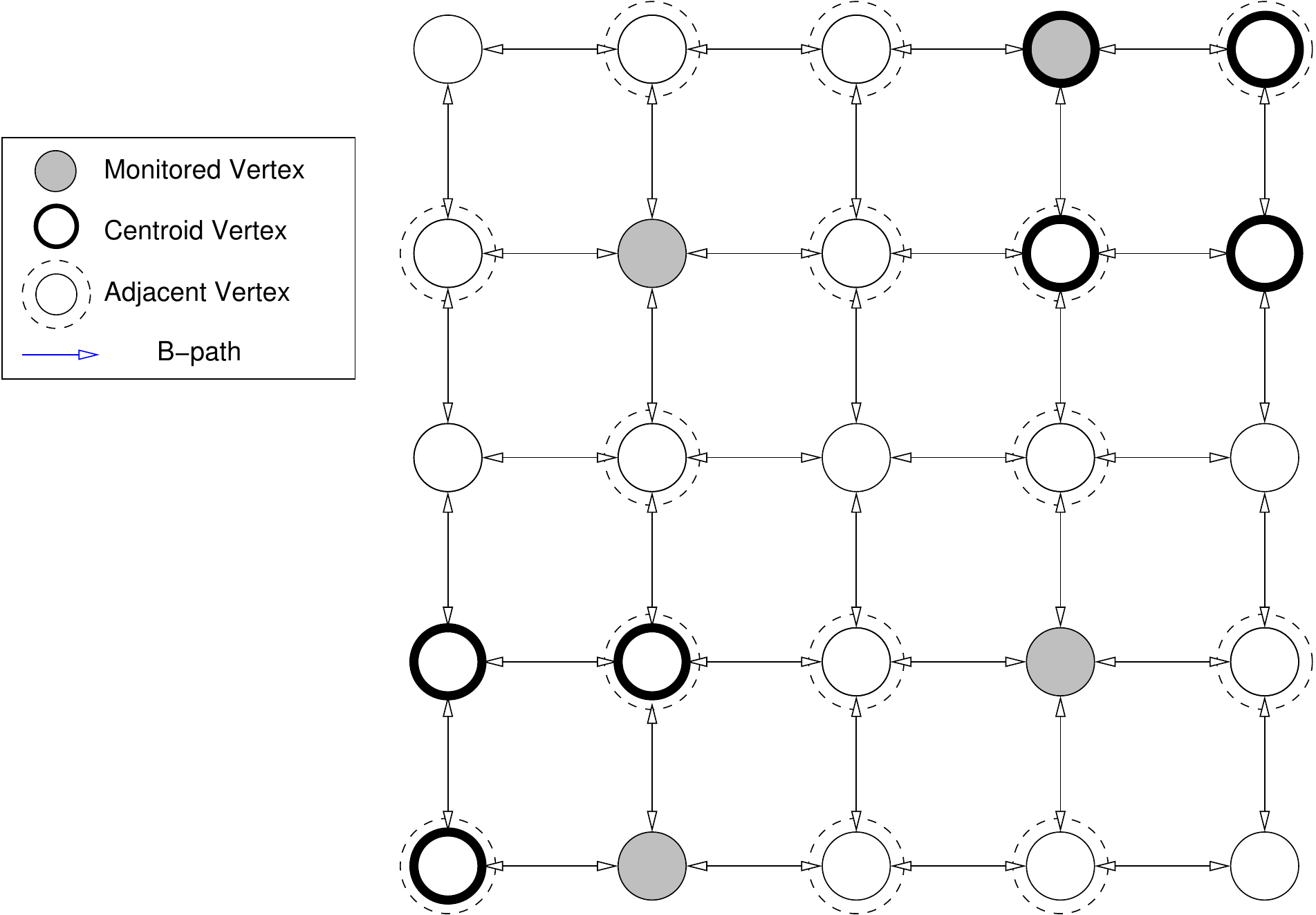}
\label{fig:grid1}
}
\subfigure[Unmonitored subgraph obtained by removing the combined cutset $C_M$]{
\includegraphics[height=2.5in]{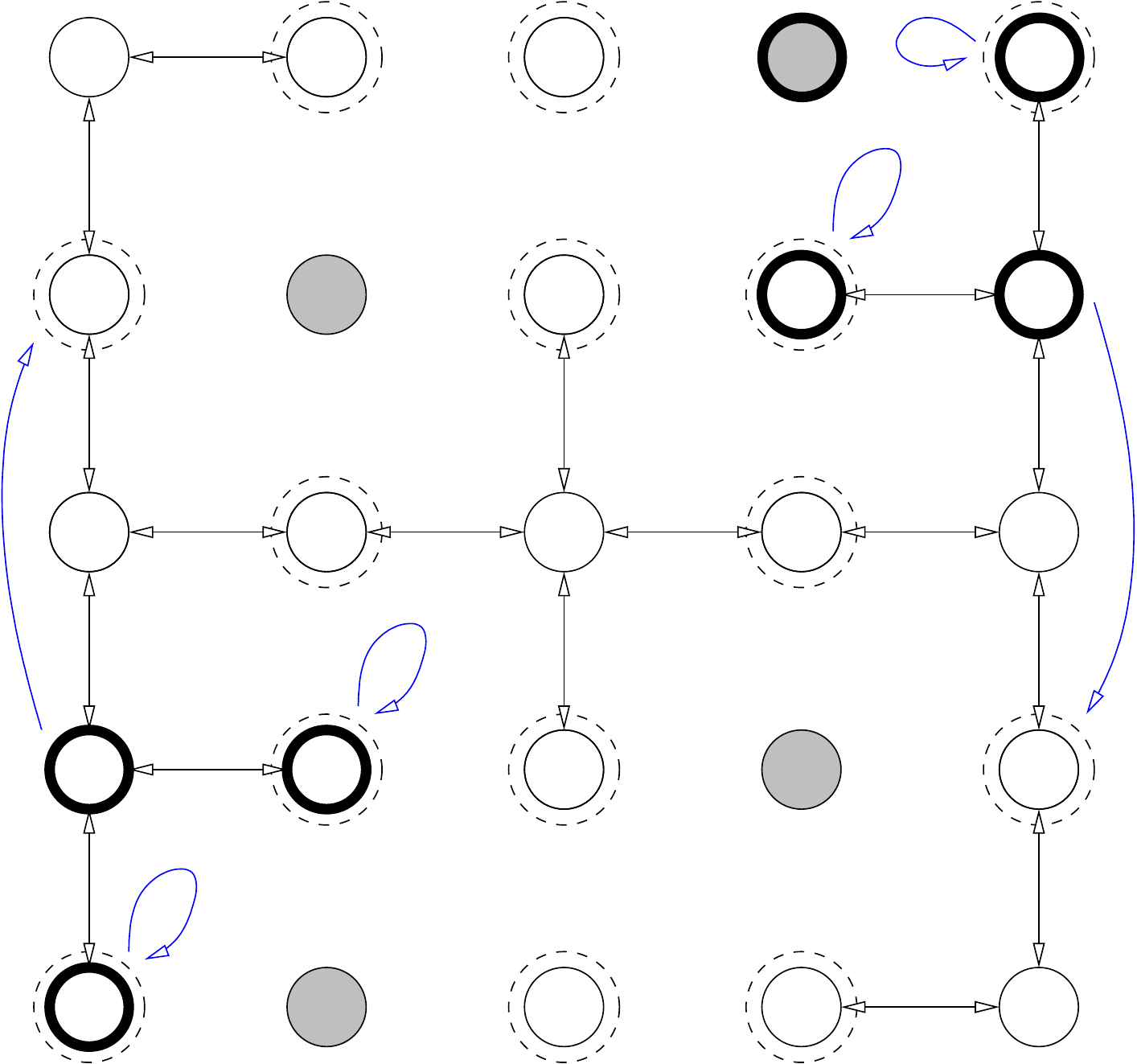}
\label{fig:grid3rem}
}
\caption{\label{fig:grid3Example}
A road network with four monitored vertices (shaded) and seven centroid vertices (bold).  This network is identical to that of Figure~\ref{fig:grid2Example} except that two of the monitored vertices have changed position.  In this modified graph, the unmonitored subgraph is a tree, and every centroid has its own B-path.  Thus, Theorem~\ref{thm:treethm} guarantees the traffic flow is fully calculable on this network.}
\end{figure}

It is worth pointing out that unmonitored subgraphs might be trees even on very large graphs with only a small fraction of monitored vertices.  We considered an $18 \times 18$ grid network (324 vertices) on which 72 vertices were monitored and all unmonitored subgraphs were trees satisfying the sufficient condition for flow calculability given by Theorem~\ref{thm:treethm}.  More generally, it could be expanded to a $3k \times 3k$ graph for any integer $k$ having $12k$ monitored vertices.  Many more examples of large traffic networks can be found for which Theorem~\ref{thm:treethm} applies.  And as we have seen in Figure~\ref{fig:grid2Example}, even when the unmonitored subgraph is not a tree, failure to satisfy the necessary condition of Theorem~\ref{thm:revslp} signals the need either to increase the number of sensors, or to rearrange their positions until the flow is calculable.

\section{Conclusions}\label{sec:conc}
We have corrected an error in an earlier theorem regarding when a proposed set of vertices $M$ in a two-way directed graph is a valid monitoring set for determining the flow on the entire graph.  The topological insights of our counterexample led us to prove a new, stronger, necessary condition that is also sufficient on any unmonitored subgraph of the network that is a tree.  We then showed by example the broad array of networks in which this condition could be used by traffic engineers to determine the placement of traffic sensors.

\bibliographystyle{hmcmath}
%\bibliographystyle{abbrv}
%\singlespace
\bibliography{thesis}

\begin{thebibliography}{2}
\providecommand{\natexlab}[1]{#1}
\providecommand{\url}[1]{\texttt{#1}}
\providecommand{\urlprefix}{URL }
\expandafter\ifx\csname urlstyle\endcsname\relax
  \providecommand{\doi}[1]{doi:\discretionary{}{}{}#1}\else
  \providecommand{\doi}{doi:\discretionary{}{}{}\begingroup
  \urlstyle{rm}\Url}\fi
\providecommand{\selectlanguage}[1]{\relax}
\providecommand{\eprint}[2][]{\url{#2}}

\bibitem[{Bianco et~al.(2006)Bianco, Confessore, and
  Gentili}]{combinatorialSLP}
Bianco, Lucio, Giuseppe Confessore, and Monica Gentili. 2006.
\newblock Combinatorial aspects of the sensor location problem.
\newblock \emph{Annals of Operations Research} 144(1):201--234.

\bibitem[{Bianco et~al.(2001)Bianco, Confessore, and Reverberi}]{odmatrixest}
Bianco, Lucio, Giuseppe Confessore, and Pierfrancesco Reverberi. 2001.
\newblock A network based model for traffic sensor location with implications
  on {O/D} matrix estimates.
\newblock \emph{Transportation Science} 35(1):50--60.

\end{thebibliography}

%\nocite{*}
%\include{biblioTransScie}
\end{document}